\documentclass{article}[11pt]
\usepackage{fullpage}
\usepackage{amsmath}
\usepackage[utf8]{inputenc}  
\usepackage{amstext}
\usepackage{amssymb}
\usepackage{amscd}
\usepackage{amsthm}
\usepackage{amsfonts}
\usepackage{graphicx}
\usepackage[utf8]{inputenc} 
\usepackage{amssymb, enumerate}
\usepackage{tikz}
	\usetikzlibrary{arrows,automata}
\usepackage{float}

\newtheorem{theorem}{Theorem}[section]
\newtheorem*{Theorem}{Theorem}
\newtheorem{lemma}[theorem]{Lemma}
\newtheorem{proposition}[theorem]{Proposition}

\theoremstyle{definition}
\newtheorem{definition}[theorem]{Definition}
\newtheorem{remark}[theorem]{Remark}

\newtheorem{example}[theorem]{Example}

\newcommand{\RR}{\mathbb{R}} \newcommand{\TT}{\mathbb{T}}  \newcommand{\ZZ}{\mathbb{Z}}   \newcommand{\QQ}{\mathbb{Q}}  \newcommand{\NN}{\mathbb{N}}  \newcommand{\II}{\mathbb{I}} \newcommand{\JJ}{\mathbb{J}}

\newcommand{\Do}{\mathcal{D}} \newcommand{\Po}{\mathcal{P}}  \newcommand{\G}{\mathcal{G}^{(n)}}

\begin{document}

\title{Minoration of the complexity function associated to a translation on the torus.}

\author{Nicolas Bédaride \footnote{LATP  UMR 7353, Universit\'e Aix Marseille, 
39 rue Joliot Curie, 
13453 Marseille cedex 13, France.
{\it E-mail address:} {\tt nicolas.bedaride@univ-amu.fr}}
\& Jean-Fran\c cois Bertazzon\footnote{LATP  UMR 7353, {\it E-mail address:} {\tt jeffbertazzon@gmail.com}}}

\date{}

\makeatletter
\renewenvironment{thebibliography}[1]
     {\section*{\refname}%
      \@mkboth{\MakeUppercase\refname}{\MakeUppercase\refname}%
      \list{\@biblabel{\@arabic\c@enumiv}}%
           {\settowidth\labelwidth{\@biblabel{#1}}%
            \leftmargin\labelwidth
            \advance\leftmargin\labelsep
            \@openbib@code
            \usecounter{enumiv}%
            \let\p@enumiv\@empty
            \itemsep=0pt
            \parsep=0pt
            \leftmargin=\parindent
            \itemindent=-\parindent
            \renewcommand\theenumiv{\@arabic\c@enumiv}}%
      \sloppy
      \clubpenalty4000
      \@clubpenalty \clubpenalty
      \widowpenalty4000%
      \sfcode`\.\@m}
     {\def\@noitemerr
       {\@latex@warning{Empty `thebibliography' environment}}%
      \endlist}
\makeatother

\maketitle

\begin{abstract}
We show that the word complexity function $p_k(n)$ of a piecewise translation map associated to a minimal translation on the torus $\TT^k = \RR^k / \ZZ^k$ is at least $kn+1$ for every integer $n$.
\end{abstract}

\section{Introduction}

Symbolic dynamics is a part of dynamics which studies the interaction between dynamical systems and combinatorial properties of sequences defined over a finite number of symbols. The collection of these symbols is called {\it the alphabet}. Here we are interested in a minimal translation on the torus $\TT^k$, {\it i.e} every point has a dense orbit under the action of the translation. We consider a finite partition of the torus, and the coding of a translation related to this partition. An orbit is coded by a one-sided infinite sequence over a finite alphabet, the sequence is called {\it an infinite word}. The complexity of an infinite word is a function defined over integers $n\in\mathbb{N}$: it is the number of different words of $n$ digits which appear in such a sequence. For a minimal dynamical system, when the partition consists of "reasonable" sets, the complexity does not depend on the orbit. For a given dynamical system, the complexity function depends on the partition. First let us recall some usual facts about translations on a torus. To define different partitions of a torus, we will proceed as follow: we consider a subset $\mathcal{D}$ of $\mathbb{R}^k$ which tiles the plane by translations. We call it a {\it fundamental domain} of the torus. Then a translation on the torus can be seen as a bijective map defined on the set $\mathcal{D}$.
Thus a partition of the torus gives a partition of $\mathcal{D}$, and the translation gives a map piecewise defined on each subset of $\mathcal{D}$.  This map is called a {\it piecewise translation associated to a torus}, see Section 2 for some examples.
Here we search partitions which minimize the complexity function for a minimal translation. In doing so, we are looking for sets $\mathcal{D}$ and partitions of this set such that the piecewise translation coded by this partition has the minimal complexity.

The case of the circle (identified to $[0,1]$) is known since the work of Hedlund and Morse: 
a translation by angle $\alpha$ can be coded by the intervals $[0,1-\alpha)$ and $[1-\alpha,1)$. Then the complexity function is equal to $n+1$ if $\alpha\notin \mathbb{Q}$, see \cite{He.Mo.40}. Moreover a sequence of complexity $n+1$ can be characterized as a coding of a translation by an irrational angle on the circle, see \cite{Co.He.73} and \cite{Pyth.02}. Thus we can not find another partition of the circle such that the complexity of the translation related to this partition is less than $n+1$.

For the two dimensional torus, the first result was made for a particular translation:
 there exists a translation, a fundamental domain which  is a fractal set and a partition of this set such that every orbit is coded by a word of complexity $2n+1$. This is a famous result of Rauzy, the partition of the torus is called the Rauzy fractal and each element has a fractal boundary, see \cite{Rau.82}. 
For any translation, some bounds on the complexity function are given if the partition and the fundamental domain are made of polygons. The first one has been made in \cite{Arno.Maud.Shio.Tamu.94} for a partition in three rhombi. This result has been generalized to any dimension in \cite{Bary.95}.
In \cite{Beda.03} a correction of the proof in the two dimensional cases was made and in \cite{Beda.09}  an improvement of the result in any dimension was done for the same type of partition. Finally, in \cite{Chev.09}, the growth of the complexity was computed for every polygonal partition in $\mathbb{T}^2$. In this case the complexity is quadratic in $n$. Thus for the translation defined by Rauzy, the lowest possible complexity is less than $2n+1$.
The same phenomena appear in higher dimension for some particular translations. In \cite{Bert.12} the second author shows that for a translation on a two dimensional torus, the complexity function related to a partition  cannot be less than $2n+1$. This bound is sharp, by Rauzy's result. This means that for a minimal translation, whatever the partition is, we can not find one with a  complexity smaller than $2n+1$. Here we prove a similar result in the $k$ dimensional torus.

In order to understand it and to find the fundamental domain with the lowest partition we define the notion of piecewise translation associated to a minimal translation in Section \ref{se:not}.  This notion means that we study a minimal translation and one partition in the same object.

\begin{Theorem}
Let $k\geq 1$ and $m\geq 1$ be two integers, let $\bold{a}$ be a vector in $\RR^k$ such that the translation by $\bold{a}$ on the torus $\TT^k$ is minimal. Let $(T,\Do_1,\ldots,\Do_m)$ be a piecewise translation associated to this translation. Then the complexity function of the piecewise translation fulfills 
$$\forall n\geq 1, \quad p_k(n)\geq kn+1.$$
\end{Theorem}

We will see in Section \ref{subse:ex} that there exists for each integer $k$ a minimal translation on the torus $\TT^k$ and a a piecewise translation associated to this translation such that the complexity function is $kn+1$ for any integer $n$. If we fix an integer $k$ and an arbitrary minimal translation on $\TT^k$, we do not know yet what  the minimal complexity function is for  a piecewise translation associated to this translation.

The complexity function is related to a topological invariant of dynamical systems: the topological entropy. This notion is used to classify systems of zero topological entropy. There are few systems for which the complexity can be computed. Recently Host, Kra and Maass give a lower bound in \cite{Host.Kra.Ma.13}, of the complexity of a class of dynamical systems called  nilsystems. These systems are close to rotations.

\section{Examples}\label{sec:example}
We begin by a list of examples which show that the bound on the complexity in the Theorem \ref{thm} is sharp. The definition of the complexity function and the definition of a piecewise isometry associated to a translation will be given in the next section. We refer to \cite{Pyth.02} for a complete background on combinatorics on words.
\subsection{Background on Substitutions}
To present the examples we need to give some background on Combinatorics on Words.
Consider an alphabet over $k$ letters $\{a_1, a_2, \dots, a_k\}$ and the free monoid $\{a_1,a_2,\dots,a_k\}^*$. A morphism of the free monoid is called a {\it substitution}. For example consider 
$$\sigma_k=\sigma :\begin{cases}a_1\mapsto a_1a_2\\ a_2\mapsto a_1a_3\\ \vdots\\ a_k\mapsto a_1.\end{cases}$$
This substitution is called {\it k-bonacci} substitution.
The sequence $(\sigma^n(a_1))_{n\in\mathbb{N}}$ converges for the product topology to a fixed point $w_k$ of $\sigma$:
$$w_k=a_1a_2a_1a_3a_1a_2a_1a_4  \ldots a_1a_ka_1 \dots, \sigma(w_k)=w_k$$

The abelianization of the monoid is $\ZZ^k$. We can represent it in $\RR^k$ with a basis where each coordinate vector represents either $a_1,\dots, a_{k-1}$ or $a_k$.
The abelianization of $\sigma$ is a linear morphism of $\mathbb{Z}^k$ with matrix
$$
M_\sigma=\left( \begin{array}{c|c}
1\ \cdots \ 1&1\\
\hline
I_k & \begin{array}{cc} 0 \\ \vdots \\ 0 \end{array}
\end{array}\right)
\text{ , where $I_k$ is the identity matrix of $\RR^{k-1}$.}
$$
The abelianization of $w_k$ is a broken line in $\mathbb{Z}^k$ denoted $\delta$.

This matrix has one real eigenvalue of modulus bigger than $1$. The eigenspace is of dimension one, and is called $H_e$. The other eigenvalues are complex numbers if $k\geq 3$, we denote by $H_c$ a real hyperplane orthogonal to $H_e$.
Consider the projection on $H_c$ parallel to $H_e$. 
Denote $\mathcal{R}_{k-1}$ the closure of the projection of vertices of $\delta$. 
This set is the {\it Rauzy fractal} associated to the substitution.
There is a natural partition of this set in $k$ subsets defined as follows: 
First the vertices of $\delta$ can be split in $k$ classes. One made of vertices followed by an edge in the $a_1$ direction, one made of vertices followed by an edge in the $a_2$ direction, and so on. 
The closure of the projection of each family of vertices gives us a subset of $\mathcal{R}_{k-1}$ denoted $\mathcal{R}_{k-1}(a_i), i=1\dots k$. 
$$\mathcal{R}_{k-1}=\displaystyle\bigcup_{i=1}^k\mathcal{R}_{k-1}(a_i).$$

\begin{theorem}\cite{Messa.98}
For every integer $k\geq 1$ the Rauzy fractal $\mathcal{R}_{k}$ is a fundamental domain of the torus $\TT^k$. 
Let $u$ be the projection of one basis vector of $\mathbb{R}^{k+1}$ on $H_c$. Then the translation by vector $u$ on the torus $\TT^k$ is minimal, and the complexity of this map related to the partition of $\mathcal{R}_{k}$ has complexity
$$p_k(n)=kn+1.$$
\end{theorem}

To rephrase the statement, there exists a piecewise translation map associated to the translation by vector $u$, and the complexity of this translation according to the partition $\mathcal{R}_{k}=\displaystyle\bigcup_{i=1}^{k+1}\mathcal{R}_{k}(a_i)$ has complexity $kn+1$.

\subsection{Examples of substitutions} \label{subse:ex}
We list some examples of piecewise translations associated to $k$-bonacci substitutions.
\begin{itemize}
\item For $k=2$, $R_1=[0,1)$, the translation is $x\mapsto x+a \quad mod\quad 1$ for $a=\varphi$ the golden mean. The partition is $D_1=[0,1-a), D_2=[1-a,1)$. The fixed point $w_2$ is called {\it Fibonacci word}.

\item For $k=3$, $w_k$ is called {\it Tribonacci word}. The partition is made of fractal sets.
\end{itemize}

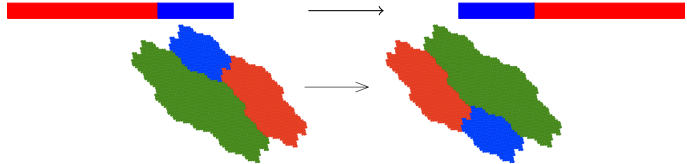
\begin{figure}[h!]
\centering
\begin{tikzpicture}
	\draw  [fill=gray,color=red!100!black,opacity=1] (0,0)  -- (2,0) -- (2,0.2) -- (0,0.2) -- (0,0);
	\draw [fill=gray,color=blue!100!black,opacity=1] (2,0) -- (3,0) -- (3,0.2) -- (2,0.2) -- (2,0);
\draw[red] (0,0)--(2,0);
\draw [blue] (2,0)--(3,0);
\draw [->] (4,0.1)--(5,0.1);
	\draw  [fill=gray,color=red!100!black,opacity=1] (7,0)  -- (9,0) -- (9,0.2) -- (7,0.2) -- (7,0);
	\draw [fill=gray,color=blue!100!black,opacity=1] (7,0) -- (6,0) -- (6,0.2) -- (7,0.2) -- (7,0);
\draw [blue] (6,0)--(7,0);
\draw[red] (7,0)--(9,0);
\end{tikzpicture}

 \includegraphics[width=6cm]{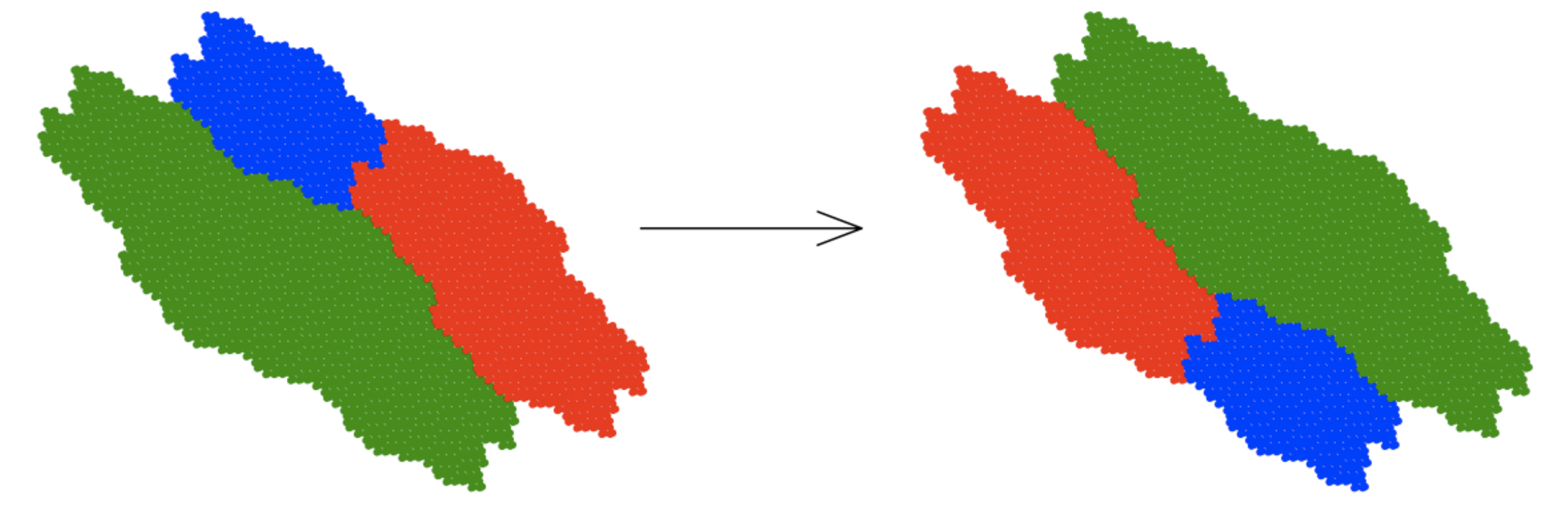}
\caption{Representation of the piecewise translation associated to the k-bonacci word for $k=2$ and $3$.}
\end{figure}

\begin{figure}[h!] \centering  \includegraphics[width=6cm]{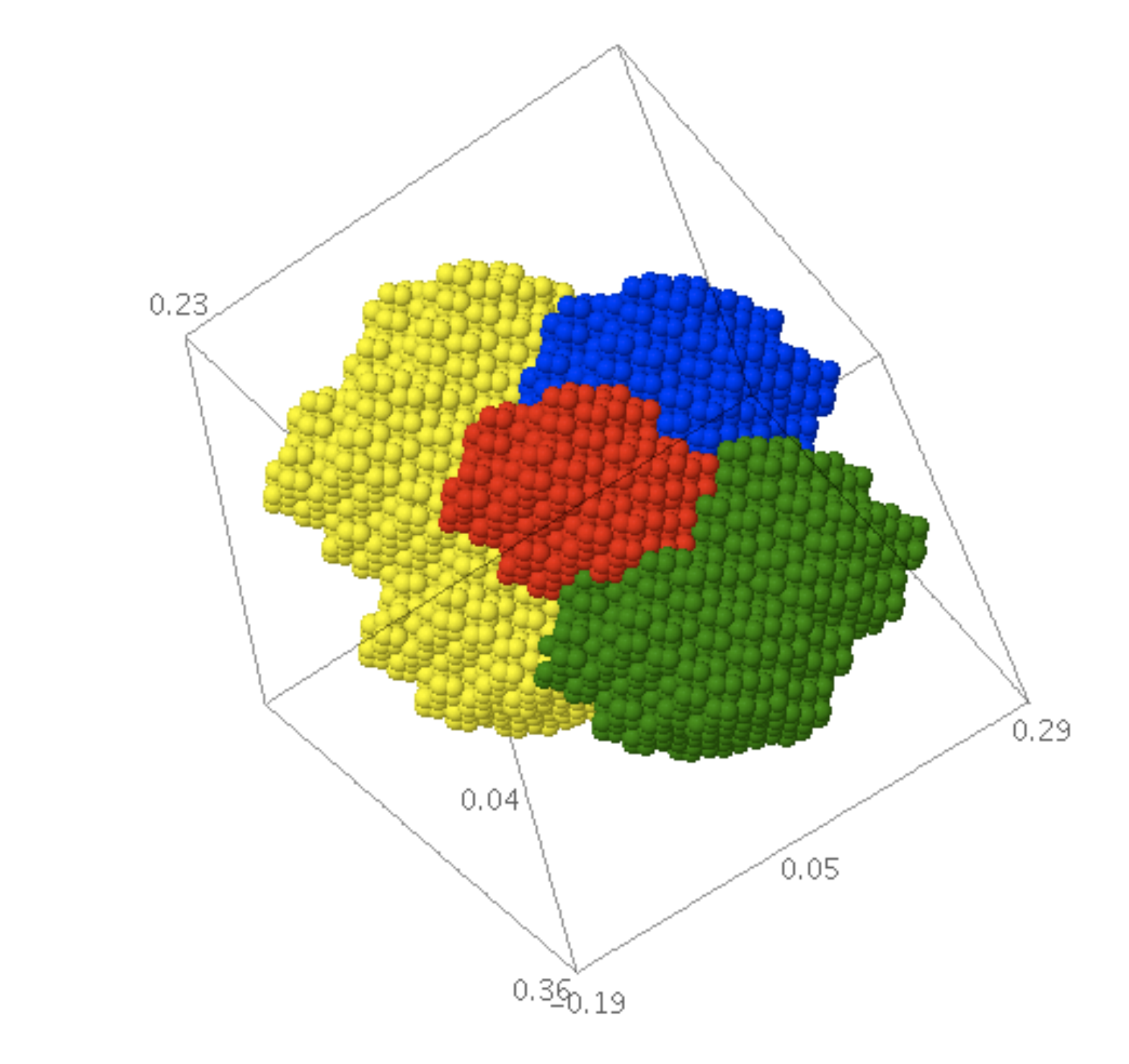}  \includegraphics[width=7cm]{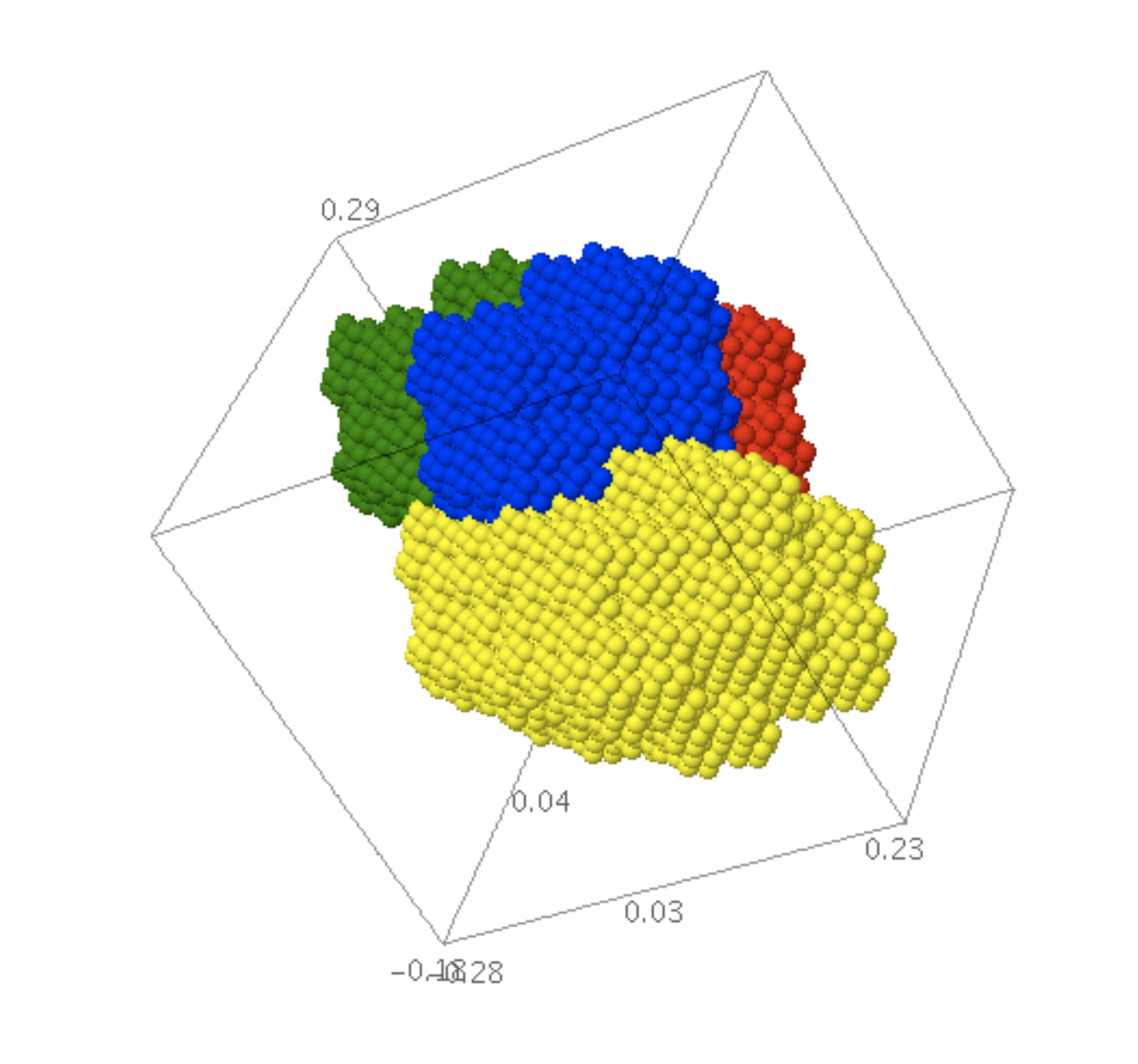}
\caption{The Rauzy fractal and the piecewise translation associated to the 4-bonacci substitution.}
\end{figure}

\subsection{Example of a translation}
Here we present an example of translation on the torus $\mathbb{T}^2$. 
Consider an hexagon with identified opposite edges. Denote $\bold{u},\bold{v}, \bold{w}$ the three vectors which support the edges of the hexagon. Then the hexagon tiles the plane with a lattice generated by $\bold{v}-\bold{u},\bold{w}-\bold{u}$.  Thus the hexagon is a fundamental domain of a torus $\mathbb{T}^2$, and there is a partition of this domain in three parallelograms. Now consider the map defined by the following figure. It is a piecewise translation, and modulo this lattice, the three translation vectors $\bold{u},\bold{v}, \bold{w}$ are equal, we can denote it by $\bold{a}$. 
Thus the piecewise translation is associated to the translation on the torus given by:
$$\begin{array}{ccc}
\mathbb{T}^2&\rightarrow&\mathbb{T}^2\\
m&\mapsto&m+\bold{a}
\end{array}$$

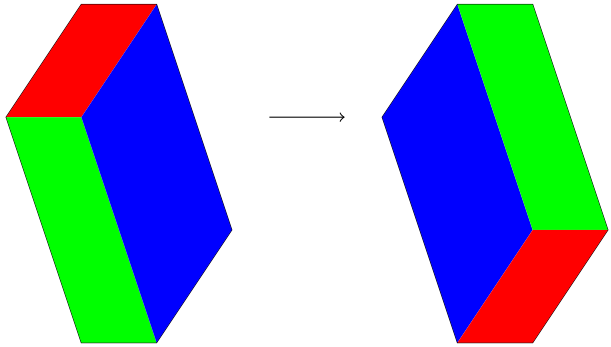
\begin{figure}
\begin{center}
\begin{tikzpicture}[scale=.5]
\draw (0,0)--(2,3)--(4,3)--(6,-3)--(4,-6)--(2,-6)--cycle;
\fill[red] (0,0)--(2,0)--(4,3)--(2,3);
\fill[green] (0,0)--(2,0)--(4,-6)--(2,-6);
\fill[blue] (2,0)--(4,3)--(6,-3)--(4,-6);

\draw[->] (7,0)--(9,0);
\draw (10,0)--(12,3)--(14,3)--(16,-3)--(14,-6)--(12,-6)--cycle;
\fill[blue] (10,0)--(12,3)--(14,-3)--(12,-6);
\fill[green] (12,3)--(14,3)--(16,-3)--(14,-3);
\fill[red] (12,-6)--(14,-6)--(16,-3)--(14,-3);
\end{tikzpicture}
\caption{Piecewise translation associated to a translation on the two dimensional torus.}
\end{center}
\end{figure}

In this case the complexity function is quadratic in $n$:
\begin{theorem}[\cite{Beda.03}]\label{thm:moi}
For almost all $\bold{a}\in\mathbb{T}^2$ the complexity function of the translation related to this partition is:
$$p_2(n)=n^2+n+1.$$
\end{theorem}
\section{Notations and statement of the result} \label{se:not}

Let $k\geq 2$ be an integer, let $\bold{a}=(a_1,\dots,a_k)\in \RR^k$ be a vector and $\lambda$  be the Lebesgue measure on $\RR^k$. The {\it translation on the torus} $\TT^k=\RR^k/ \ZZ^k$ by $\bold{a}$ is the map
$$ \begin{array}{ll} \TT^k & \rightarrow \TT^k  \\ \bold{x} & \mapsto \bold{x}+\bold{a}.\end{array}$$

This map is {\it minimal} (every orbit is dense) if and only if 
$$\forall q_1,\ldots,q_k, \ \forall q\in \mathbb{Q}, \ q_1a_1+\cdots+q_ka_k=q\Longrightarrow q_1=\dots=q_k=0.$$

For a minimal translation, Lebesgue measure is the unique invariant probability measure and then, it is an ergodic measure.

A {\it fundamental domain} of the torus $\TT^k=\RR^k/ \ZZ^k$ is a subset $\mathcal{D}$ of $\mathbb{R}^k$ which periodically tiles the space $\RR^k$ by the action of $\ZZ^k$, except maybe on a subset of $\RR^k$ with measure zero.

\begin{definition} 
A {\it piecewise translation map} $(T,\Do_1,\ldots,\Do_m)$ {\it associated} to a translation  by the vector  $\bold{a}$ is a map defined on $\Do=\displaystyle\bigcup_{i=1}^m\Do_i$ where $\Do_1,\ldots,\Do_m$ are $m$ measurable disjoint sets  of $\mathbb{R}^k$, such that for each $\bold{x}\in \mathcal{D}$: $T(\bold{x}) = \bold{x}+\bold{a}+\bold{n}(\bold{x})$ where:
\begin{itemize}
\item $\bold{n} :\mathcal D \mapsto \ZZ^k$ is a measurable map,
\item $\Do=\displaystyle\bigcup_{i=1}^m\Do_i$ is a fundamental domain of the torus,
\item for each integer $i\in\{1,\ldots,m\}$, there exists a vector $\bold{r}_i\in \QQ^k$ such that:
	\begin{equation} \label{eq:defri}
	\int _{\Do_i} \bold{n} (\bold{x})\  \mathsf{d} \lambda(\bold{x}) = \lambda(\Do_i)\cdot  \bold{r}_i .
	\end{equation}
\end{itemize}
\end{definition}

The {\it coding} of the orbit of a point $\bold{x}\in \Do$ under $T$ is defined by
$$\mathsf{Cod}:\Do\mapsto \{1,\cdots,m\}^\NN$$ 
$$\mathsf{Cod}(\bold{x})_n = i \Longleftrightarrow T^n\bold{x}\in \Do_i.$$

\begin{remark}
Examples given in preceding Section show that a rotation can have different combinatorics depending on the fundamental domain of the torus. This is why we introduce preceding definition. It is different from the definition of a partition, due to the last points.  
First the domain of $\Do$ is not assumed to be bounded. Here we do not assume that $\mathsf{Cod}$ is an injective map, see the work of Halmos for a general reference about the coding \cite{Hal.56}. Moreover our result holds in particular for some piecewise translations for which the dynamical  symbolic system is not conjugate to the translation on the torus. The terminology "piecewise translation" is not perfect since we allow multiple vectors of translation in each subset of the fundamental domain. Indeed,  the map $ \bold{n}$ can take an infinity of values. However, in most interesting examples, the partition of the fundamental domain is such that there is only one translation vector in each piece. 
\end{remark} 

A word of length $n$ in $\mathsf{Cod}\left( \bold{x}\right)$ is a finite sequence of the form $(\mathsf{Cod}\left( \bold{x}\right)_m)_{p\leq m \leq p+n-1}.$
The {\it complexity function} of the piecewise translation $(T,\Do_1,\ldots,\Do_m)$ is a map from $\NN^*$ to $\NN$ which  associates to every integer $n\geq 1$, the number of different words of lengths $n$ in $\mathsf{Cod}\left( \bold{x}\right)$. This number is independent of the point $\bold{x}$ for a minimal rotation. Our principal result is the following:

\begin{theorem}\label{thm}
Let $k\geq 1$ and $m\geq 1$ be two integers, let $\bold{a}$ be a vector in $\RR^k$ such that the translation by $\bold{a}$ on the torus $\TT^k$ is minimal. Let $(T,\Do_1,\ldots,\Do_m)$ be a piecewise translation associated to this translation. Then the complexity function of the piecewise translation fulfills 
$$\forall n\geq 1, \quad p_k(n)\geq kn+1.$$
\end{theorem}

The proof is based on the following propositions proved in Section \ref{subse1} and Section \ref{subse2}.

\begin{proposition} \label{prop1}
Let $k\geq 1$ and $m\geq 1$ be two integers, $\bold{a}$ a vector in $\RR^k$ such that the translation by $\bold{a}$ on the torus $\TT^k$ is minimal. Let $(T,\Do_1,\ldots,\Do_m)$ be a piecewise translation associated to this translation. 
Then we have: 
$$m\geq k+1.$$
\end{proposition}

\begin{proposition} \label{prop2}
Let $k\geq 1$ and $m\geq 1$ be two integers, $\bold{a}$ a vector of $\RR^k$ such that the translation by $\bold{a}$ on the torus $\TT^k$ is minimal. Let $(T,\Do_1,\ldots,\Do_m)$ be a piecewise translation associated to this translation. 
Then the complexity function fulfills for every integer $n$:
$$p_k(n+1)-p_k(n)\geq k.$$
\end{proposition}

The lower bound on $p_k(n)$ in Theorem \ref{thm} is optimal, see Section \ref{sec:example} for examples.

\section{Background on Graph theory}\label{sec:graph}

Before the proof of the result we give some background on Graph theory. This part will be used in the proof of Proposition \ref{prop2}. The reference for this section is \cite{Boll.98}.

Let $G$ be an oriented graph with $n$ vertices $\mathsf{V}=(\mathsf{v}_1,\ldots,\mathsf{v}_n)$ and $m$ edges $\mathsf{E}=(\mathsf{e}_1,\ldots,\mathsf{e}_m)$, we assume that $E$ is a subset of $V\times V$.
Let $C_1(G)$ be the vector space generated by maps from edges to $\mathbb{R}$. 

We will be interested by the functions $f\in C_1(G)$ such that in each vertex $\mathsf{v}$,
\begin{equation}\label{equgraph}
\sum \limits_{\mathsf{v}' \text{such that } \mathsf{vv}' \in \mathsf{E} } f(\mathsf{vv}') = \sum \limits_{\mathsf{v}' \text{such that } \mathsf{v'v} \in \mathsf{E} } f(\mathsf{v'v}).
\end{equation}

These functions generate a vector space denoted by $N(G)$. 

An oriented cycle $\mathsf{L}$ is defined by $k$ vertices $\mathsf{u}_1,\ldots,\mathsf{u}_k$ in $\mathsf{V}$ such that $\mathsf{u}_1=\mathsf{u}_k$ and for each $1\leq i\leq k-1$, $\mathsf{u}_i\mathsf{u}_{i+1}$ or $\mathsf{u}_{i+1}\mathsf{u}_i$ is an edge of the graph. We write $\mathsf{L}=\mathsf{u}_1\cdots \mathsf{u}_k$.
 
Let $\mathsf{L}=\mathsf{u}_1\cdots \mathsf{u}_k$ be an oriented cycle in $G$, then we define a function $z_{\mathsf L}$ on $E$ by
$$
z_{\mathsf L}({\mathsf e})= \begin{cases}
1 \text{ if there exist $j\in \{1,\ldots,k-1\}$ such that }  \mathsf{e} = \mathsf{u}_{j} \mathsf{u}_{j+1} \\ 
-1 \text{ if there exist $j\in \{1,\ldots,k-1\}$ such that }  \mathsf{e} = \mathsf{u}_{j+1} \mathsf{u}_{j}\\ 
0\quad \text{otherwise} 
\end{cases}
$$
The functions $z_{\mathsf{L}}$ is called a {\it cycle vector} and we denote by $Z(G)$ the space generated by $z_{\mathsf{L}}$ when $\mathsf{L}$ runs over cycles of $G$. It is a subspace of $N(G)$.

\begin{theorem}\label{thm:graph}
Let $G$ be a connected graph with $n$ vertices and $m$ edges. 
The space $Z(G)$ is equal to $N(G)$, and
the dimension of the vector space $Z(G)$ is given by:
 $$dim(Z(G))=m-n+1.$$
\end{theorem}
The elements of a basis of $Z(G)$ are called {\it fundamental cycles}.
We refer to \cite{Boll.98} for a proof of this result.

\begin{example}
Consider the following directed graph. The space $Z(G)$ is of dimension $6-4+1=3$. 
To the cycle $132$ is associated the function $z_{132}$ represented by $e_6-e_2-e_1$.
For example, an element of $N(G)$ is the function $f$ given by
$f(e)=\begin{cases}
1\quad e=e_1\\
0\quad e=e_2\\
1\quad e=e_3\\
2\quad e=e_4\\
1\quad e=e_5\\
1\quad e=e_6\\
\end{cases}$ .
For example, by looking at vertex $1$ we have: $2=1+1$.
\begin{center}
\begin{tikzpicture}[scale=1]

\node[below] (A) at(0,0){$1$};
\node[below] (B) at(2,0){$2$};
\node[above] (C) at(2,2){$3$};
\node[above] (D) at(0,2){$4$};

\path (A) edge[bend right,->] (B);
\path (B) edge[bend right,->] (C);
\path (C) edge[bend right,->] (D);
\path (D) edge[bend right,->] (A);

\draw[->] (0,0)--(2,2);
\draw [->] (2,0)--(0,2);

\node[below] at(1,-1){$e_1$};
\node[right] at(2.5,1){$e_2$};
\node[above] at(1,2.5){$e_3$};
\node[left] at(-.5,1){$e_4$};
\node[below] at(1.5,1.5){$e_6$};
\node[above] at(.5,1.5){$e_5$};

\end{tikzpicture}
\end{center}
\end{example}
\section{Proofs of Proposition \ref{prop1} and Proposition \ref{prop2}}

\subsection{A first lemma}
\begin{lemma} \label{lem}
Let $k,m\geq 1$ be two integers, let $\bold{a}=\begin{pmatrix}a_1\\\vdots\\ a_k\end{pmatrix}$ be a vector of $\RR^k$,  $(\bold{n}_i)_{1\leq i \leq m}$ $m$ vectors of $\QQ^k$ and let $\alpha_1,\ldots\alpha_m$ be $m$ real numbers such that 
$$\bold{a} = \alpha_1\bold{n}_1 + \cdots \alpha_m \bold{n}_m.$$ 
Assume $m< k$,  then there exists $k$ rational numbers $q_1,\ldots,q_k$, non all equal to zero, such that 
$$a_1q_1+\cdots a_k q_k=0.$$
\end{lemma}

\begin{proof}
We prove the result by induction on $k$. 
For $k=2$, we have $m=1$ and the hypothesis gives:

\[
\begin{pmatrix} a_1 \\ a_2 \end{pmatrix} = \alpha_1 \begin{pmatrix} {n}_1(1) \\ {n}_1(2) \end{pmatrix} .
\]
Either ${n}_1(1)$ or ${n}_1(2)$ are non zero numbers and we obtain $a_2 {n}_1(1)- a_1  {n}_1(2)=0$, or ${n}_1(1)={n}_1(2)=0$. Both cases give $\bold{a}=\bold{0}$ and the result is proved.

Assume the result is true until $k\geq2$. The case $m=1$ is similar to the preceding case: it suffices to consider the two first coordinates. Now assume $m>1$, for $1\leq i \leq k$, the $i$-coordinate of the vectors gives: 
\[
a_i = \alpha_1n_1(i) + \cdots \alpha_m n_m(i).
\]

Up to permutation, we can assume that $n_1(1)$ is non zero (if either $\bold{a}=\bold{0}$ or if there exists $i$ such that $a_i=0$ the result is obvious).
By linear combination we have for some $2\leq i \leq k$ :
\[
n_1(1) a_i - n_1(i) a_1= \alpha_2  \left| \begin{array}{cc} n_1(1)&n_2(1) \\n_1(i)&n_2(i)\\ \end{array}\right|+ \cdots +\alpha_m \left| \begin{array}{cc} n_m(1)&n_m(1) \\n_1(i)&n_2(i)\\ \end{array}\right|.
\]

By induction hypothesis, there exist $(k-1)$ rational numbers not all equal to zero $q_2,\ldots q_{k}$ such that 
\[
q_2\big{(} n_1(1) a_2 - n_1(2) a_1 \big{)} + \cdots+ q_k \big{(} n_1(1) a_k - n_1(k) a_1 \big{)} = 0
\]

Let us define $q_1 =  -\frac{q_2n_1(2)+\cdots+q_m n_1(m)}{n_1(1)}$.
Then previous equation can be written as
\[
q_1 a_1 + q_2 a_2+ \cdots q_k a_k=0. 
\]

The induction hypothesis is proven.

\end{proof}

\subsection{Proof of Proposition \ref{prop1}} \label{subse1}
Let $k\geq 1$ be an integer, $\bold{a}$ a vector of $\RR^k$ such that the translation by $\bold{a}$ on the torus $\TT^k$ is minimal. Let $(T,\Do_1,\ldots,\Do_m)$ be a piecewise translation associated to this translation. The proof is made by contradiction. Assume $m\leq k$.
For every integer $i\in\{1,\ldots,m\}$, let us denote $A_i=\lambda(D_i)$ the volume of $\Do_i$. 
For a function $\bold{f}=(f_1,\ldots,f_k)\in$ L$^1_\lambda(\Do,\RR^k)$, Birkhoff's theorem can be applied to the ergodic map $T$ and gives for almost every point $x$:
$$\lim_{N\to +\infty}\frac{1}{N}\displaystyle\sum_{k=0}^{N-1} \bold{f} \left(T^k \bold{x}\right)=\left( \int_{\Do} f_1 (\bold{x})d\lambda(\bold{x}) , \ldots , \int_{\Do} f_k (\bold{x})d\lambda(\bold{x}) \right).$$
We say that $x$ is \textit{generic} for $f$ if the formula is true for $x$.
Moreover we say that $x$ is \textit{recurrent} if for a neighborhood $V$ of $x$, there exists an integer sequence $(N_p)_{p\in\mathbb{N}}$ such that $T^{N_p}(x)$ belongs to $V$.
Now let $x$ be a recurrent and generic for the functions $\bold{n} \cdot 1_{\Do_i}$, for $i\in \{1,\ldots,m\}$ (the measure of such points is one). 
Let $N$ be an integer, we obtain 
\begin{equation} \label{eq:iteration}
T^N (\bold{x}) = \bold{x}+N\bold{a}+\sum_{k=0}^{N-1}\bold{n}\left(T^k \bold{x}\right)=\bold{x}+N\bold{a}+\sum_{i=1}^m\sum_{k=0}^{N-1}\bold{n}\left(T^k \bold{x}\right)1_{\Do_i}\left(T^k\bold{x}\right). 
\end{equation}
Since $\bold{x}$ is a recurrent point for $T$, there exists an integer sequence $(N_p)_{p\in \NN}$ such that
$T^{N_p}(\bold{x})/N_p$ converges to zero when $p$ tend to infinity. By Birkhoff theorem we obtain in Equation \eqref{eq:iteration} :
$$\frac{T^{N_p}(\bold{x})}{N_p} =\frac{\bold{x}}{N_p}+\bold{a}+\sum_{i=1}^m\frac{1}{N_p}\sum_{k=0}^{N_p-1}\bold{n}\left(T^k \bold{x}\right)1_{\Do_i}\left(T^k\bold{x}\right). $$
And when $N_p$ tends to infinity, we get:
\begin{equation} \label{eq1}
0 = \bold{a} + A_1 \bold{r}_1+ \cdots + A_m  \bold{r}_m,
\end{equation}
where the vectors $\bold{r}_i$ are defined in \eqref{eq:defri}. 
Since the set $\Do$ is of volume $1$ we have also:
\begin{equation} \label{eq2}
1=A_1+\cdots+A_m.
\end{equation}

Thus Equation \eqref{eq1} can be written 
$0 = \bold{a}' + A_2  \bold{r}_1'+ \cdots + A_m  \bold{r}_m',$
where $\bold{r}_i'=\bold{r}_i-\bold{r}_1$ and $\bold{a}'=\bold{a}+\bold{r}_1$. By Lemma \ref{lem} we deduce that if $m\leq k$, the translation by vector $\bold{a}'$ can not be a minimal translation, thus the translation by vector $\bold{a}$ can not be a minimal translation, this is a contradiction. 

\subsection{Proof of Proposition \ref{prop2}} \label{subse2}

Let $(T,\Do_1,\ldots,\Do_{p_k(1)})$ be a piecewise translation associated to the minimal translation by $\bold{a}$ on the torus $\TT^k$ and let $n$ be an integer.
To each step $n\geq 1$, the $(n-1)$-th refinement of the initial partition consists of $p_k(n)$ domains denoted by $\Do_1^{(n)},\ldots,\Do_{p_k(n)}^{(n)}$. The measures of domains are respectively  denoted by $A_1^{(n)},\ldots,A_{p_k(n)}^{(n)}$. We write $\Po^{(n)}$ the partition generated by the domains $\Do_1^{(n)},\ldots,\Do_{p_k(n)}^{(n)}$ .

We fix an integer $n\geq 2$ until the end of  the proof.

We define a graph $\G$ as follows :
\begin{itemize}
\item To each domain $\Do_i^{(n)}$, $i\in \{1,\ldots,p_k(n)\}$, is associated a vertex. 
\item For each integer $\ell\in \{1,\ldots,p_k(n+1)\}$, we define an oriented edge $\mathsf{a}_\ell^{(n)}$ from the vertex $\Do_i^{(n)}$ to the vertex $\Do_j^{(n)}$ if there exists two integers $i$ and $j$ in $\{1,\ldots,p_k(n)\}$ such that 
	\[
	\Do_\ell^{(n+1)} \subset \Do_i^{(n)} \mbox{ and } T^{-1}\big{(}\Do_\ell^{(n+1)}\big{)}\subset \Do_j^{(n)}.
	\]
	
\end{itemize}

This graph is called the $n$-th {\it Rauzy Graph} associated to the language of the piecewise translation, see Section \ref{rex}.
By assumption, each domain has a non zero Lebesgue measure, and Lebesgue measure is ergodic for the translation, thus the graph $\G$ is connected. Consider the function $f\in C_1(\G)$ which associates to any edge $\mathsf{a}_\ell^{(n)}$ the value $ A_l^{(n+1)}$. The applications $T$ and $T^{-1}$ are Lebesgue measure preserving transformations, so the function $f$ is in the space $N(\G)$, thus in the space $Z(\G)$ by Theorem \ref{thm:graph}. 
We deduce that there exists $\chi=p_k(n+1)-p_k(n)+1$ fundamental cycles in this graph such that the function $f\in Z(G_n)$ is a linear combination of fundamental cycles.
Thus there exists some real numbers $\alpha_1,\ldots,\alpha_\chi$ such that for every integer $\ell$ there exists a subset  $\II_\ell$  of $\{1,\ldots,\chi\}$ such that
\begin{equation} \label{eq3}
A_\ell^{(n+1)} = \sum _{i\in \II_\ell} \alpha_i.
\end{equation}

The partition $\Po^{(n+1)}$ is a refinement of the initial partition $\Po^{(1)}$. So for any integer $\ell \in \{1,\ldots,p(1)\}$, there exists a subset $\JJ_\ell$ of  $\{1,\ldots,p_k(n+1)\}$ such that
\begin{equation} \label{eq4}
\Do_\ell = \bigcup \limits_{m \in \JJ_\ell} \Do_m^{(n+1)} \mbox{ , and so } A_\ell^{(1)} =  \sum _{m \in \JJ_\ell} A_{m}^{(n+1)}.
\end{equation}

Using Relations \eqref{eq3} and \eqref{eq4} in Equation \eqref{eq1}, we find :
\[
\begin{array}{llll}
0 &= \bold{a} + A_1^{(1)} \bold{r}_1+ \cdots +  A_{p(n)} ^{(1)}  \bold{r}_{p(1)} \\
 &= \bold{a} + \left( \sum \limits_{m \in \JJ_1} A_{m}^{(n+1)} \right)  \bold{r}_1+ \cdots + \left( \sum \limits_{m \in \JJ_{p(n+1)}} A_{m}^{(n+1)} \right)  \bold{r}_{p(1)} \\
 &= \bold{a} + \left( \sum \limits_{m \in \JJ_1}   \sum \limits_{i\in \II_m} \alpha_i \right) \bold{r}_1+ \cdots + \left( \sum \limits_{m \in \JJ_{p(n+1)}}   \sum \limits_{i\in \II_m} \alpha_i \right) \bold{r}_{p(1)}.
\end{array}
\]
We can reorganize this sum to get the relation:
\[
0 = \bold{a} +  \alpha_1\cdot \bold{s}_1+ \cdots +  \alpha_\chi \cdot \bold{s}_m,
\]
where each vector $\bold{s}_i$ is a linear combination with integer coefficients of  some vectors $\bold{r}_j$, and so is a vector of $\QQ^k$.
By Lemma \ref{lem},  we conclude that $\chi \geq k-1$, and so $p_k(n+1)-p_k(n)\geq k$. 

\subsection{Proof of Theorem \ref{thm}}
Proposition \ref{prop1} shows that $p_k(1)\geq k+1$. Now Proposition \ref{prop2} shows that for every integer $n$ we have $p_k(n+1)-p_k(n)\geq k$. By induction we deduce that 
$$p_k(n)\geq kn+1.$$

\subsection{Example}
The proof of Proposition \ref{prop2} can be illustrated by the following example:

\begin{figure}[h]
\begin{center}
\begin{tikzpicture}[scale=1]
\node  at(4.1,1){$\mathcal{D}_1^{(1)}$};
\node  at(6,1){$\mathcal{D}_2^{(1)}$};
\node  at(4.4,-1){$\mathcal{D}_3^{(1)}$};

\draw (2,-1)--(4,2)--(6,3)--(8,2)--(6,-1)--(4,-2)--cycle;
\draw (2,-1)--(4,0)--(6,3);
\draw (4,0)--(6,-1);

\draw (9,-1)--(11,2)--(13,3)--(15,2)--(13,-1)--(11,-2)--cycle;

\draw[dashed] (9,-1)--(11,0)--(13,3);
\draw[dashed]  (11,0)--(13,-1);
\draw[dashed]  (11,0)--(10,-1.5);
\draw[dashed]  (11,0)--(10,0.5);
\draw[dashed]  (11,0)--(13,1)--(14,2.5);
\draw[dashed]  (13,1)--(14,0.5);
\node at (11.3,1.3) {$\mathcal{D}_1^{(2)}$};
\node at (12.8,1.7) {$\mathcal{D}_2^{(2)}$};
\node at (14.1,1.5) {$\mathcal{D}_3^{(2)}$};
\node at (12.8,0) {$\mathcal{D}_4^{(2)}$};
\node at (11.5,-1) {$\mathcal{D}_5^{(2)}$};
\node at (9.9,-0.9) {$\mathcal{D}_6^{(2)}$};
\node at (10.2,-0.05) {$\mathcal{D}_7^{(2)}$};
\end{tikzpicture}
\end{center}
\end{figure}

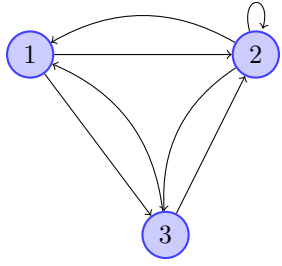
\begin{figure}[h]
\begin{center}
\begin{tikzpicture}[scale=0.6]
\tikzstyle{place}=[circle,thick,draw=blue!75,fill=blue!20,minimum
                      size=6mm]

\node[place] (A) at(5,-5){$1$};
\node[place] (B) at(10,-5){$2$};
\node[place] (C) at(8,-9){$3$};

\path (A) edge[->] (B);
\path (B) edge[bend right,->] (A);
\path (A) edge[->] (C);
\path (C) edge[bend right,->] (A);
\path (B) edge[bend right,->] (C);
\path (C) edge[->] (B);
\path (B) edge[loop above] (B);

\end{tikzpicture}
\caption{Example of the first Rauzy graph associated to this piecewise translation.}
\end{center}
\end{figure}

Consider the example described in Section \ref{sec:example}.
On the left the is the partition in three sets. Denote them by $\mathcal{D}_1^{(1)}$, $\mathcal{D}_2^{(1)}$ and $\mathcal{D}_{3}^{(1)}$. Consider points such that the image by the piecewise translation is in one of these sets. 
To understand them we need to consider the second refinement of the partition. We draw it on the right. It is made by $2^2+2+1=7$ pieces denoted $\mathcal{D}_1^{(2)},\ldots,\mathcal{D}_{7}^{(2)}$. Each set has an image included in a piece of the left partition. This allows us to build the first of Rauzy graph of this piecewise translation map.

\subsection{Rauzy Graphs} \label{rex}

We have given a definition of the Rauzy graphs in Section \ref{subse2} with a "geometric" point of view but we can also introduce them with a purely combinatorial point of view. We refer to \cite{Che.Hub.Mess.01, Ca.97}. Cassaigne studies the graphs of $w_2$ in \cite{Ca.97}. Both definitions are similar. We can deduce from our result :

\begin{proposition}[Corollary of Theorem \ref{thm}]
The Rauzy graphs of the language of a piecewise translation associated to the minimal translation by $\bold{a}$ on the torus $\TT^k$ have an Euler characteristic at least $k$.
\end{proposition}

\bibliographystyle{plain}
\bibliography{biblio-mino}

\begin{thebibliography}{10}

\bibitem{Arno.Maud.Shio.Tamu.94}
P.~Arnoux, C.~Mauduit, I.~Shiokawa, and J.~Tamura.
\newblock Complexity of sequences defined by billiard in the cube.
\newblock {\em Bull. Soc. Math. France}, 122(1):1--12, 1994.

\bibitem{Bary.95}
Yu. Baryshnikov.
\newblock Complexity of trajectories in rectangular billiards.
\newblock {\em Comm. Math. Phys.}, 174(1):43--56, 1995.

\bibitem{Bert.12}
J.~F. Bertazzon.
\newblock Fonction complexité associée à une translation ergodique du tore.
\newblock {\em Bulletin of the London Mathematical Society.}, 44
  (6):1155--1168, 2012.

\bibitem{Boll.98}
B.~Bollob{\'a}s.
\newblock {\em Modern graph theory}, volume 184 of {\em Graduate Texts in
  Mathematics}.
\newblock Springer-Verlag, New York, 1998.

\bibitem{Beda.03}
N.~Bédaride.
\newblock Billiard complexity in rational polyhedra.
\newblock {\em Regul. Chaotic Dyn.}, 8(1):97--104, 2003.

\bibitem{Beda.09}
N.~Bédaride.
\newblock Directional complexity of the hypercubic billiard.
\newblock {\em Discrete Math.}, 309(8):2053--2066, 2009.

\bibitem{Ca.97}
J.~Cassaigne.
\newblock Complexit\'e et facteurs sp\'eciaux.
\newblock {\em Bull. Belg. Math. Soc. Simon Stevin}, 4(1):67--88, 1997.
\newblock Journ\'ees Montoises (Mons, 1994).

\bibitem{Che.Hub.Mess.01}
N.~Chekhova, P.~Hubert, and A.~Messaoudi.
\newblock Propri\'et\'es combinatoires, ergodiques et arithm\'etiques de la
  substitution de {T}ribonacci.
\newblock {\em J. Th\'eor. Nombres Bordeaux}, 13(2):371--394, 2001.

\bibitem{Chev.09}
N.~Chevallier.
\newblock Coding of a translation of the two-dimensional torus.
\newblock {\em Monatsh. Math.}, 157(2):101--130, 2009.

\bibitem{Co.He.73}
E.~M. Coven and G.~A. Hedlund.
\newblock Sequences with minimal block growth.
\newblock {\em Math. Systems Theory}, 7:138--153, 1973.

\bibitem{Hal.56}
P.~R. Halmos.
\newblock {\em Lectures on ergodic theory}.
\newblock Publications of the Mathematical Society of Japan, no. 3. The
  Mathematical Society of Japan, 1956.

\bibitem{Host.Kra.Ma.13}
B.~Host, B.~Kra, and Maass A.
\newblock Complexity of nilsystems and systems lacking nilfactors.
\newblock {\em Arxiv 1203.3778.}, 2013.

\bibitem{Messa.98}
A.~Messaoudi.
\newblock Propri\'et\'es arithm\'etiques et dynamiques du fractal de {R}auzy.
\newblock {\em J. Th\'eor. Nombres Bordeaux}, 10(1):135--162, 1998.

\bibitem{He.Mo.40}
M.~Morse and G.~A. Hedlund.
\newblock Symbolic dynamics {II}. {S}turmian trajectories.
\newblock {\em Amer. J. Math.}, 62:1--42, 1940.

\bibitem{Pyth.02}
N.~Pytheas~Fogg.
\newblock {\em Substitutions in dynamics, arithmetics and combinatorics},
  volume 1794 of {\em Lecture Notes in Mathematics}.
\newblock Springer-Verlag, Berlin, 2002.
\newblock Edited by V. Berth{\'e}, S. Ferenczi, C. Mauduit and A. Siegel.

\bibitem{Rau.82}
G.~Rauzy.
\newblock Nombres alg\'ebriques et substitutions.
\newblock {\em Bull. Soc. Math. France}, 110(2):147--178, 1982.

\end{thebibliography}

\end{document}